\documentclass[11pt]{article}

\usepackage{amsfonts}
\usepackage{amssymb,bbm}
\usepackage{xr}
\usepackage{graphicx,psfrag}
\usepackage{paralist,subfigure,float}
\usepackage{amsmath}
\usepackage{color,soul,textcomp}
\usepackage[active]{srcltx}
\usepackage[dvipsnames]{xcolor}

\def\2{C^{1,2}(\R\times\R^N)}
\def\to{\rightarrow}
\def\e{\varepsilon}

\def\R{\mathbb{R}}
\def\N{\mathbb{N}}

\def\.{\cdot}

\def\1{\mathbbm{1}}

\newlength{\textlarg}

\newcommand{\be}{\begin{equation}}
\newcommand{\ee}{\end{equation}}
\newcommand{\baa}{\begin{array}}
\newcommand{\eaa}{\end{array}}
\newcommand{\ba}{\begin{eqnarray}}
\newcommand{\ea}{\end{eqnarray}}

\newtheorem{thm}{\bf Theorem}[section]
\newtheorem{lem}[thm]{\bf Lemma}

\newtheorem{cor}[thm]{\bf Corollary}
\newtheorem{defi}[thm]{\bf Definition}

\newenvironment{proof}{{\noindent{\emph{Proof.}}}}{\hfill$\Box$\vskip7pt}
\newenvironment{formula}[1]{\begin{equation}\label{eq:#1}}{\end{equation}\noindent}
\def\Fi#1{\begin{formula}{#1}}
\def\Ff{\end{formula}\noindent}

\numberwithin{equation}{section}

\begin{document}
\date{}
\title{\bf{Diameters of the level sets for reaction-diffusion equations in nonperiodic slowly varying media}\thanks{This work has been carried out in the framework of the A*MIDEX project (ANR-11-IDEX-0001-02), funded by the ``Investissements d'Avenir" French Government program managed by the French National Research Agency (ANR). The research leading to these results has also received funding from the ANR project RESISTE (ANR-18-CE45-0019).}}
\author{Fran\c cois Hamel$^{\hbox{\small{ a}}}$ and Gr\'egoire Nadin$^{\hbox{\small{ b}}}$\\
\\
\footnotesize{$^{\hbox{a }}$Aix Marseille Univ, CNRS, Centrale Marseille, I2M, Marseille, France}\\
\footnotesize{$^{\hbox{b }}$Sorbonne Universit{\'e}, CNRS, Universit\'{e} de Paris, Inria, Laboratoire Jacques-Louis Lions, Paris, France}}
\maketitle

\begin{abstract} 
We consider in this article reaction-diffusion equations of the Fisher-KPP type with a nonlinearity depending on the space variable $x$, oscillating slowly and non-periodically. We are interested in the width of the interface between the unstable steady state $0$ and the stable steady state $1$ of the solutions of the Cauchy problem. We prove that, if the heterogeneity has large enough oscillations, then the width of this interface, that is, the diameter of some level sets, diverges linearly as $t\to +\infty$ along some sequences of times, while it is sublinear along other sequences. As a corollary, we show that under these conditions generalized transition fronts do not exist for this equation.
\end{abstract}


\section{Introduction and main results}\label{sec1}

This article investigates propagation phenomena for heterogeneous reaction-diffusion equations:
\begin{equation} \label{eq:main}
\partial_{t}u = \partial_{xx}u +f(x,u),\ \ t>0,\ x\in\R.
\end{equation}
We consider real-valued functions $f:(x,s)\mapsto f(x,s)$ which are uniformly continuous in~$\R\times[0,1]$, Lipschitz-continuous with respect to $s\in[0,1]$ uniformly in $x\in \R$, for which there exists $\delta>0$ such that $\partial_sf$ exists and is continuous and bounded in $\R\times[0,\delta]$, and which satisfy the Fisher-KPP property
\begin{equation} \label{hyp:fKPP}\left\{\baa{l}
\forall\,x\in\R,\ \ f(x,0)=f(x,1)=0,\vspace{3pt}\\
\displaystyle\forall\,0<s_1<s_2\le1,\ \inf_{x\in\R}\Big(\frac{f(x,s_1)}{s_1}-\frac{f(x,s_2)}{s_2}\Big)>0.\eaa\right.
\end{equation}
We denote
\be\label{defmu}\left\{\baa{l}
\mu (x):= \partial_sf(x,0)\hbox{ for }x\in\R,\vspace{3pt}\\
\displaystyle\mu_{+}:=\sup_{x\in \R}\mu(x),\ \  \mu_{-}:=\inf_{x\in \R}\mu(x),\eaa\right.
\ee
and we assume that
\be\label{mu-}
\mu_->0,
\ee
implying that $0$ is a linearly unstable steady state of the ordinary differential equation $\dot\xi(t)=f(x,\xi(t))$, for each point $x\in\R$. No periodicity or any other specific condition on the dependence of $f$ with respect to $x$ is assumed.

The aim of this article is to quantify the width of the transition between the stable steady state~$1$ and the unstable steady state~$0$. More precisely, we want to construct nonlinearities for which the solutions of the Cauchy problem associated with appropriate initial data have an interface between~$1$ and~$0$ whose size diverges linearly as $t\to +\infty$ along some sequences of times. 

We consider measurable initial conditions $u_0$ satisfying:
\begin{equation} \label{hyp:u0}
u_0(x)=0 \hbox{ for all } x>0, \quad \lim_{x\to -\infty}u_0(x)=1, \quad 0\leq u_{0}\leq 1\hbox{ in }\R.
\end{equation}

The Cauchy problem~\eqref{eq:main} with initial condition~\eqref{hyp:u0} is well posed: the unique bounded solution~$u$ belongs to $\mathcal{C}^{1,2}_{t,x}((0,+\infty)\times\R)$, $u(t,\cdot)\to u_0$ as $t\to0^+$ in $L^1_{loc}(\R)$, $0<u<1$ in $(0,+\infty)\times\R$ and
$$\left\{\baa{l}
u(t,-\infty)=1,\vspace{3pt}\\
u(t,+\infty)=0,\eaa\right.\hbox{locally uniformly with respect to $t\ge0$}.$$
For $\gamma\in (0,1)$ and $t>0$, we can then define 
$$\left\{\baa{rcll}
X^{+}_\gamma(t) & \!\!:=\!\! & \max\big\{ x\in \R:u(t,x)\geq\gamma\big\} & \!\!\!=\max\big\{ x\in \R:u(t,x)=\gamma\big\},\vspace{3pt}\\
X^{-}_\gamma(t) & \!\!:=\!\! & \min\big\{ x\in \R:u(t,x)\leq\gamma\big\} & \!\!\!=\min\big\{ x\in \R:u(t,x)=\gamma\big\},\vspace{3pt}\\ 
I_{\gamma}(t) & \!\!:=\!\! & X^{+}_\gamma(t)-X^{-}_\gamma(t)\,\ge\,0.&\eaa\right.$$
Notice immediately that, if $f$ does not depend on $x$ and if $u_0$ were also assumed to be nonincreasing, then $\mu$ would be constant and $u(t,\cdot)$ would be continuous and decreasing in $\R$ for every $t>0$, hence $X^+_\gamma(t)=X^-_\gamma(t)$ and $I_\gamma(t)=0$ for every~$\gamma\in(0,1)$ and $t>0$.

The main result of the paper shows that, for some functions $f$ as above, the diameters~$I_\gamma(t_n)$ of the level sets associated with small values $\gamma>0$ can grow linearly with respect to $t_n$ along some diverging sequences of times $(t_n)_{n\in\N}$, while $I_\gamma(t'_n)$ is small with respect to $t'_n$ along some other diverging sequences of times $(t'_n)_{n\in\N}$, for any fixed value~$\gamma\in(0,1)$.

\begin{thm}\label{th1}
There are some functions $f:\R\times[0,1]\to\R$, fulfilling the conditions~\eqref{hyp:fKPP}-\eqref{mu-} and
\be\label{mu+2mu-}
\mu_+>2\mu_->0,
\ee
such that the solution $u$ of~\eqref{eq:main} and~\eqref{hyp:u0} satisfies
$$\limsup_{t\to +\infty} I_{\gamma}(t)=+\infty$$
for all $\gamma>0$ small enough, and even 
\be\label{main1}
\liminf_{\gamma\to 0^{+}}\Big(\limsup_{t\to +\infty} \frac{I_{\gamma}(t)}{t}\Big)\geq\frac{\mu_+ }{\sqrt{\mu_{+}-\mu_{-}}}-2\sqrt{\mu_-}>0.
\ee
On the other hand, for all $\gamma\in(0,1)$, there holds
\be\label{main2}
0=\liminf_{t\to +\infty} \frac{I_{\gamma}(t)}{t}\le\limsup_{t\to +\infty} \frac{I_{\gamma}(t)}{t}\le2(\sqrt{\mu_+}-\sqrt{\mu_-})<+\infty.
\ee
Furthermore, $f$ can be chosen of the type $f(x,s)=\mu(x)s(1-s)$, where $\mu$ is a $C^\infty(\R)$ function with bounded derivatives at any order.
\end{thm}

Several comments are in order. First of all, an easy computation shows that the condition $\mu_+>2\mu_->0$ ensures that
\be\label{mupm0}
\frac{\mu_+ }{\sqrt{\mu_{+}-\mu_{-}}}-2\sqrt{\mu_-}>0,
\ee
hence the positivity of $\limsup_{t\to+\infty}I_\gamma(t)/t$ for all $\gamma>0$ small enough. 

We also point out that the functions $f$ constructed in the proof of Theorem~\ref{th1} are such that~$\mu(x)=\partial_sf(x,0)$ take values arbitrarily close to the mini\-mal and maximal values $\mu_-$ and $\mu_+$ in larger and larger intervals as $x\to+\infty$. 

Let us now make the link between Theorem~\ref{th1} and the notion of generalized transition fronts connecting $1$ to~$0$ for~\eqref{eq:main}, defined as follows:

\begin{defi}\label{def:gtf} {\rm{\cite{BH,S}}}
A generalized transition front connecting $1$ to $0$ for equation~\eqref{eq:main} is an entire solution $U:\R\times\R\to(0,1)$ associated with a real valued map $t\mapsto\xi_t$ such that
$$U(t,\xi_t+x)\to1\hbox{ $($resp.~$0)$ as $x\to-\infty$ $($resp. $x\to+\infty)$ uniformly with respect to $t\in\R$.}$$ 
\end{defi}

We can then derive from Theorem \ref{th1} the non-existence of generalized transition fronts. 

\begin{cor}\label{cor1}
There are some functions $f:\R\times[0,1]\to\R$, fulfilling the conditions~\eqref{hyp:fKPP}-\eqref{mu-} and~\eqref{mu+2mu-}, such that equation~\eqref{eq:main} does not admit any generalized transition front connecting~$1$ to~$0$. 
\end{cor}

Zlatos \cite{Z} proved that, if $\mu_{+}<2\mu_{-}$, then there exists a generalized transition front connecting~$1$ to~$0$ for this equation. It follows from the arguments developed in the proof of Corollary \ref{cor1} that, in that case, for  $u_{0}\equiv\1_{(-\infty,0)}$, the width $I_{\gamma}(t)$ of the interface remains bounded with respect to~$t>0$, for any fixed $\gamma\in(0,1)$. We leave as an open problem the critical case $\mu_+=2\mu_-$. 

Up to our knowledge, the optimality of the condition $\mu_+<2\mu_-$ for the existence of transition fronts connecting $1$ to $0$ had been proved only for compactly supported perturbations of a homogeneous nonlinearity $f(u)$~\cite{NRRZ}. Therefore, our result shows the optimality of this condition $\mu_+<2\mu_-$ on the existence of transition fronts connecting $1$ to $0$ for a new class of heterogeneous nonlinearities $f(x,u)$, and it also provides some quantitative estimates of the diameters of the level sets of the solution $u$ of~\eqref{eq:main} and~\eqref{hyp:u0} at large time.

In a recent paper, Cerny, Drewitz and Schmitz~\cite{CDS} investigate similar questions in a different situation, namely that of random stationary ergodic media. The authors consider a random growth rate $\mu (x, \omega)$ depending on a random event~$\omega$, such that $\mu$ is ergodic with respect to spatial shifts and such that the law of $\mu (x,\cdot)$ does not depend on $x\in \R$. It is proved in~\cite{CDS} that, if $\mu_+>2\mu_-$, then $I_{\gamma}(t)=X_{\gamma}^{+}(t)-X_{\gamma}^{-}(t)$ grows at least logarithmically as $t\to +\infty$. The location of $X_\gamma^\pm(t)$ is also addressed in a companion paper~\cite{DS}. 

Next, several authors addressed the question of spreading speeds for Fisher-KPP reaction-diffusion equations during the last decade: they tried to determine whether the quantities $X_{\gamma}^{\pm}(t)/t$ converge as $t\to +\infty$. For example, the case of initial data which oscillate between two decreasing exponential functions and which lead to convergent or non-convergent quantities $X^\pm_\gamma(t)/t$ has been addressed in~\cite{HN,Y}, when $f$ does not depend on $x$. For initial data decaying more slowly than any exponential function, then acceleration occurs: $\lim_{t\to+\infty}X_{\gamma}^{\pm}(t)/t=+\infty$ (this has been proved when $f$ does not depend on $x$ in~\cite{HR,H}, and when $f$ is periodic in $x$ in \cite{H}). On the other hand, we refer to \cite{BN} for a review of situations where $X_{\gamma}^{\pm}(t)/t$ converges as $t\to +\infty$. 

The type of nonlinearities we use in order to prove Theorem \ref{th1} was initially introduced in~\cite{GGN} in order to quantify the spreading speeds. In~\cite{GGN}, the authors showed that, under the assumptions~\eqref{hypxn}-\eqref{defmu2} below, one has 
\be\label{liminfX-}
\liminf_{t\to +\infty} \frac{X_{\gamma}^{-}(t)}{t} = 2\sqrt{\mu_{-}},
\ee
for every $\gamma\in(0,1)$. From some inequalities used in~\cite{GGN}, the equality $\liminf_{t\to +\infty}X_{\gamma}^{+}(t)/t=2\sqrt{\mu_-}$ can also be derived (even if this equality was not written explicitly in~\cite{GGN}). We actually prove here this latter equality as a consequence of other estimates (see the end of the proof of Theorem~\ref{th1} in Section~\ref{sec2}). It was also proved in~\cite{GGN} that, under some additional assumptions, then $\limsup_{t\to +\infty}X_{\gamma}^-(t)/t=\limsup_{t\to +\infty}X_{\gamma}^+(t)/t= 2\sqrt{\mu_{+}}$,\footnote{Together with~\eqref{liminfX-}, this shows that the upper bound of $\limsup_{t\to+\infty}I_\gamma(t)/t$ in~\eqref{main2} is optimal in general.} and other conditions were provided ensuring that, on the contrary, $X_{\gamma}^{\pm}(t)/t$ converges as $t\to+\infty$. 

Here, our point of view is different since our aim is to measure the width of the interface $I_{\gamma}(t)=X_{\gamma}^{+}(t)-X_{\gamma}^{-}(t)$. On the one hand, from the observations of the previous paragraph, we get that $\liminf_{t\to +\infty}I_{\gamma}(t)/t=0$ for every $\gamma\in(0,1)$. But this yields no further information on~$I_{\gamma}(t)$, and other arguments are needed in order to estimate $I_\gamma(t)/t$ and to bound it from below by positive constants along some diverging sequences of times.

As explained above, our result shows the non-existence of generalized transition fronts connec\-ting~$1$ to~$0$ for some nonlinearities $f$ satisfying~\eqref{hyp:fKPP}-\eqref{mu-} and $\mu_{+}>2\mu_{-}$. However, another notion of front has been introduced by the second author in \cite{N}, that of critical traveling waves. These waves are time-global solutions that are steeper than any other one (see the proof of Corollary~\ref{cor1} for more details on the notion of steepness). Critical traveling waves exist under very mild conditions, that are satisfied by \eqref{hyp:fKPP}-\eqref{mu-}. The description of these critical traveling waves is an interesting open question which goes beyond the scope of this article and which we leave open for a future work.


\section{Proofs of Theorem~\ref{th1} and Corollary~\ref{cor1}}\label{sec2}

We start with the proof of Theorem~\ref{th1}. Throughout the proof, we fix two positive constants $\mu_\pm$ such that
\be\label{mupm}
\mu_+>2\mu_->0.
\ee
We also consider any two increasing sequences $(x_{n})_{n\in\N}$ and $(y_{n})_{n\in\N}$ of real numbers such that
\begin{equation}\label{hypxn}\left\{\begin{array}{l}
0<x_n<y_n-2<y_n<x_{n+1}\ \hbox{ for all }n\in\N,\vspace{3pt}\\
y_n-x_n\to+\infty\ \hbox{ and }\ y_n=o(x_{n+1})\ \hbox{ as }n\to+\infty.\end{array}\right.
\end{equation}
Notice that~\eqref{hypxn} also implies that $x_n\to+\infty$, $y_n\to+\infty$, $x_{n+1}-y_n\to+\infty$ and $y_n=o(x_{n+1}-y_n)$ as $n\to+\infty$.
A typical example is given by $x_n=(2n+3)!$ and $y_n=(2n+4)!$ for every $n\in\N$. Other examples are given by $x_n=n!$ and $y_n=n!+\alpha\,n^\beta$ or $y_n=n!+\alpha\ln n$ (for $n$ large enough), with $\alpha>0$, $\beta>0$.

We then fix a uniformly continuous function $\mu:\R\to\R$ such that
\be\label{defmu2}\left\{\baa{ll}
\mu_-\le\mu(x)\le\mu_+ & \hbox{for all $x\in\R$},\vspace{3pt}\\
\mu (x)=\mu_{+} & \hbox{if }x\in[x_{n}+1,y_{n}-1],\vspace{3pt}\\
\mu (x)=\mu_{-} & \hbox{if }x\in[y_{n},x_{n+1}],\eaa\right.
\ee
and we set
\be\label{deff}
f(x,s)=\mu(x)\,s\,(1-s)\ \hbox{ for }(x,s)\in\R\times[0,1].
\ee
This function $f$ then satisfies~\eqref{hyp:fKPP}-\eqref{mu-} and the general regularity properties listed in the introduction. Furthermore, the function $\mu$ can be chosen of class $C^\infty(\R)$ with bounded derivatives at any order. 

The proof of Theorem~\ref{th1} consists in showing that the solution $u$ of~\eqref{eq:main} and~\eqref{hyp:u0} with this function~$f$ satisfies the desired conclusions listed in Theorem~\ref{th1}. For the proof of~\eqref{main1}, the strategy is to estimate~$u$ from above and below at some larger and larger times and at some further and further points. More precisely, on the one hand, the fact that $\mu$ takes its maximal value $\mu_+$ on the large intervals $[x_n+1,y_n-1]$ lead to the existence of time-growing bumps on these spatial intervals, whereas, on the other hand, the fact that $\mu$ takes its minimal value $\mu_-$ on the large intervals $[y_n,x_{n+1}]$ slows down the propagation in some time intervals when the growing bumps are still negligible (see the following lemmas for further details and Figure~1 below).

First of all,~\eqref{mupm} yields~\eqref{mupm0}, that is,
\be\label{ineqmupm}
\mu_+>2\sqrt{\mu_-(\mu_+-\mu_-)},
\ee
and there is then $\e_0\in(0,\mu_-)\subset(0,\mu_+/2)$ small enough so that, for every $\e\in(0,\e_0]$,
\be\label{defepsilon}
0<\frac{2\mu_+-\e-2\mu_-}{2(\mu_+-2\e)\sqrt{\mu_+-\mu_-}}<\frac{2\sqrt{\mu_+-\mu_-}}{\mu_++2\sqrt{\mu_-(\mu_+-\mu_-)}}<\frac{1}{2\sqrt{\mu_-}}.
\ee

Consider any $\e\in(0,\e_0]$. We then choose $R>0$ such that
\be\label{defReps}
\frac{\pi^2}{4R^2}\le\e<\mu_-<\frac{\mu_+}{2}
\ee
and $\Gamma\in(0,1)$ such that
\be\label{gammaeps}
f(x,s)\ge(\mu(x)-\e)s\ge(\mu_--\e)s\ \hbox{ for all }x\in\R\hbox{ and }s\in[0,\Gamma].
\ee
From~\eqref{defmu2}-\eqref{deff}, it turns out that we can choose $\Gamma=\e/\mu_+$.

Our goal is to show that
$$\limsup_{t\to+\infty}\frac{I_\gamma(t)}{t}\ge\frac{2(\mu_+-2\e)\sqrt{\mu_+-\mu_-}}{2\mu_+-\e-2\mu_-}-2\sqrt{\mu_-}>0\ \hbox{ for all $\gamma\in(0,\Gamma]$}.$$
This will then yield the desired conclusion~\eqref{main1} in Theorem~\ref{th1}, due to the arbitrariness of~$\e\in(0,\e_0]$.

We start with a lemma extending~\cite[Lemma~3.3]{NRRZ} to nonlinearities $f(x,s)$ that are not linear near $s=0$, and providing some lower Gaussian estimates in semi-infinite intervals on the right-hand side of points moving with speed $2\sqrt{\mu_--\e}$.

\begin{lem}\label{lem:estdecay}
There exists $\theta>0$ such that, for every $\gamma\in(0,\Gamma]$ and every bounded solution~$v$ of~\eqref{eq:main} with $\gamma\,\1_{(-1,0)}\le v(0,\cdot)\le1$ in $\R$, there holds
\begin{equation}\label{eq:estdecay}
v(t,x) \geq\theta\,\gamma\,e^{(\mu_{-}-\e)t}\int_{-1}^{0}\frac{e^{-(x-z)^{2}/(4t)}}{\sqrt{4\pi t}}\,dz
\end{equation}
for all $t>0$ and $x\geq 2\sqrt{\mu_{-}-\e}\,t$. 
\end{lem}

\begin{proof}
Let $V$ denote the bounded solution of
$$\partial_tV=\partial_{xx}V+\mu_-V(1-V)\ \hbox{ in }(0,+\infty)\times\R$$
with initial condition $V(0,\cdot)=\Gamma\,\1_{(-1,0)}$ in $\R$. From the maximum principle, one has~$0\le V\le1$ in $[0,+\infty)\times\R$ and the standard spreading result of~\cite{AW} implies that~$\max_{|x|\le ct}|V(t,x)-1|\to0$ as~$t\to+\infty$ for every $c\in[0,2\sqrt{\mu_-})$. In particular, there is~$T>0$ such that
\be\label{vepsgamma}
V(t,2\sqrt{\mu_--\e}\,t)\ge\Gamma\ \hbox{ for all $t\ge T$}.
\ee
Denote
$$\theta=e^{-\mu_-T}\ \in(0,1)$$
and let us check that the conclusion of Lemma~\ref{lem:estdecay} holds with this constant $\theta$. 

To do so, consider now any $\gamma\in(0,\Gamma]$ and any bounded solution $v$ of~\eqref{eq:main} such that~$\gamma\,\1_{(-1,0)}\le v(0,\cdot)\le1$ in $\R$. Define
$$v_\gamma=\frac{\gamma}{\Gamma}\,V.$$
One has $0\le v_\gamma(0,\cdot)=\gamma\,\1_{(-1,0)}\le v(0,\cdot)\le1$ in $\R$ together with $0\le v_\gamma\le V\le1$ in~$[0,+\infty)\times\R$ and
$$\partial_tv_\gamma=\partial_{xx}v_\gamma+\frac{\gamma}{\Gamma}\mu_-V(1-V)\le\partial_{xx}v_\gamma+\mu_-v_\gamma(1-v_\gamma)\le\partial_{xx}v_\gamma+f(x,v_\gamma)$$
in $(0,+\infty)\times\R$. It then follows from the maximum principle that $v_\gamma\le v$ in $[0,+\infty)\times\R$. In particular,~\eqref{vepsgamma} yields
\be\label{vgamma}
v(t,2\sqrt{\mu_--\e}\,t)\ge\gamma\ \hbox{ for all $t\ge T$}.
\ee

Next, let us consider the bounded solution $w$ of
$$\partial_tw-\partial_{xx}w=(\mu_--\e)w\ \hbox{ in }(0,+\infty)\times\R,$$
with initial condition $w(0,\cdot)=\theta\,\gamma\,\1_{(-1,0)}$. The function $w$ is nonnegative in $[0,+\infty)\times\R$, and it is positive and exactly equal to the right-hand side of~\eqref{eq:estdecay} in $(0,+\infty)\times\R$. It remains to show that $w(t,x)\le v(t,x)$ for all $t>0$ and $x\ge2\sqrt{\mu_--\e}\,t$. First of all, since $f(x,s)\ge0$ for all~$(x,s)\in\R\times[0,1]$ and since $\gamma\,\1_{(-1,0)}\le v(0,\cdot)\le1$ in $\R$, there holds
$$v(t,x)\ge\gamma\int_{-1}^0\frac{e^{-(x-z)^{2}/(4t)}}{\sqrt{4\pi t}}dz$$
for all $t>0$ and $x\in\R$. Since $\theta\,e^{(\mu_--\e)t}=e^{-\mu_-T+(\mu_--\e)t}\le1$ for all $t\in[0,T]$, it follows that
\be\label{wv}\baa{rcl}
w(t,x) & = & \displaystyle\theta\,\gamma\,e^{(\mu_--\e)t}\int_{-1}^0\frac{e^{-(x-z)^{2}/(4t)}}{\sqrt{4\pi t}}\,dz\vspace{3pt}\\
& \le & \displaystyle\gamma\int_{-1}^0\frac{e^{-(x-z)^{2}/(4t)}}{\sqrt{4\pi t}}\,dz\ \le\ v(t,x)\ \hbox{ for all }(t,x)\in(0,T]\times\R.\eaa
\ee
Secondly, for all $t>0$,
\be\label{wgamma}\baa{rcl}
0\ <\ w(t,2\sqrt{\mu_--\e}\,t) & = & \displaystyle\theta\,\gamma\,e^{(\mu_--\e)t}\int_{-1}^0\frac{e^{-(2\sqrt{\mu_--\e}\,t-z)^2/(4t)}}{\sqrt{4\pi t}}\,dz\vspace{3pt}\\
& \le & \displaystyle\theta\,\gamma\int_{-1}^0\frac{e^{-z^2/(4t)}}{\sqrt{4\pi t}}\,dz\ \le\ \theta\,\gamma\ \le\ \gamma\ \le\ \Gamma,\eaa
\ee
hence
\be\label{wv2}
w(t,2\sqrt{\mu_--\e}\,t)\le\gamma\le v(t,2\sqrt{\mu_--\e}\,t)\ \hbox{ for all $t\ge T$},
\ee
by using~\eqref{vgamma}. Lastly, since $w$ satisfies $w(t,2\sqrt{\mu_--\e}\,t)\le\Gamma$ for all $t>0$ by~\eqref{wgamma} and since~$x\mapsto w(t,x)$ is positive, symmetric with respect to $-1/2$ and decreasing with respect to~$|x+1/2|$ for every $t>0$, one has $0<w(t,x)\le\Gamma$ for all $t>0$ and $x\ge2\sqrt{\mu_--\e}\,t$. Remembering the definition~\eqref{gammaeps} of $\Gamma$, it follows that
$$\partial_tw(t,x)\le\partial_{xx}w(t,x)+f(x,w(t,x))\ \hbox{ for all $t>0$ and $x\ge2\sqrt{\mu_--\e}\,t$}.$$
Together with~\eqref{wv2} and~\eqref{wv}, the latter applied on $\{T\}\times[2\sqrt{\mu_--\e}\,T,+\infty)$, this means that~$w$ is a subsolution of the problem satisfied by $v$ in the domain $\{(t,x)\in\R^2:t\ge T,\ x\ge2\sqrt{\mu_--\e}\,t\}$. The maximum principle then implies that $v(t,x)\ge w(t,x)$ for all $t\ge T$ and $x\ge2\sqrt{\mu_--\e}\,t$. Together with~\eqref{wv}, this yields the desired result~\eqref{eq:estdecay}. The proof of Lemma~\ref{lem:estdecay} is thereby complete.
\end{proof}

Next, let us remember that, from the standard results of~\cite{AW,F,KPP} and since $f(x,s)\ge\mu_-s(1-s)$ for all $(x,s)\in\R\times[0,1]$ (with $\mu_->0$), the solution $u$ of~\eqref{eq:main} and~\eqref{hyp:u0} is such that
\be\label{uto1}
\max_{x\le ct}|u(t,x)-1|\mathop{\longrightarrow}_{t\to+\infty}0\ \hbox{ for all }0\le c<2\sqrt{\mu_-}.
\ee
In particular, $u(t,x)\to1$ as $t\to+\infty$, for each $x\in\R$. Now, for each $n\in\N$, let $y_n$ be the positive real number given in~\eqref{hypxn}. Since $u_0(x)=0$ for all $x>0$, the function $u$ is actually continuous with respect to $(t,x)$ in $[0,+\infty)\times(0,+\infty)$, hence each map $t\mapsto u(t,y_n)$ is continuous in $[0,+\infty)$. As~$u(0,y_n)=0<1=u(+\infty,y_n)$, one can then define, for each $n\in\N$ and $\gamma\in(0,1)$, the smallest~$t_{n,\gamma}>0$ such that
\be\label{tngamma}
u(t_{n,\gamma},y_n)=\gamma.
\ee

The next lemma, which follows from Lemma~\ref{lem:estdecay}, is based on the construction of a time-increasing subsolution with compact support in a finite interval on the right-hand side of the point $x_{n+1}+1$, taking advantage of the fact that $\mu$ is equal to its maximal value~$\mu_+$ there. That will provide a lower bound of $X^+_\gamma(t_{n,\gamma}+\tau_n)$, for some large $\tau_n$ as $n\to+\infty$ (see Figure~1).

\begin{lem} \label{lem:X+}
Let $R>0$ and $\Gamma\in(0,1)$ be given in~\eqref{defReps}-\eqref{gammaeps}. For every $\gamma\in(0,\Gamma]$, there holds
$$X^{+}_\gamma(t_{n,\gamma}+\tau_n)\geq x_{n+1}+1+R\ \hbox{ for all $n$ large enough},$$
for some $\tau_n>0$ such that
\be\label{eqtaun}
\tau_n\sim\frac{2\mu_+-\e-2\mu_-}{2(\mu_+-2\e)\sqrt{\mu_+-\mu_-}}\times(x_{n+1}-y_n)\ \hbox{ as $n\to+\infty$}.\footnote{We recall that $0<\e<\mu_-<\mu_+/2<\mu_+$, hence $\mu_+-2\e>0$, $2\mu_+-\e-2\mu_->0$ and the factor in front of $(x_{n+1}-y_n)$ in~\eqref{eqtaun} is positive.}
\ee
\end{lem}

\begin{figure}[h]
\centerline{\includegraphics[width=0.7\textwidth]{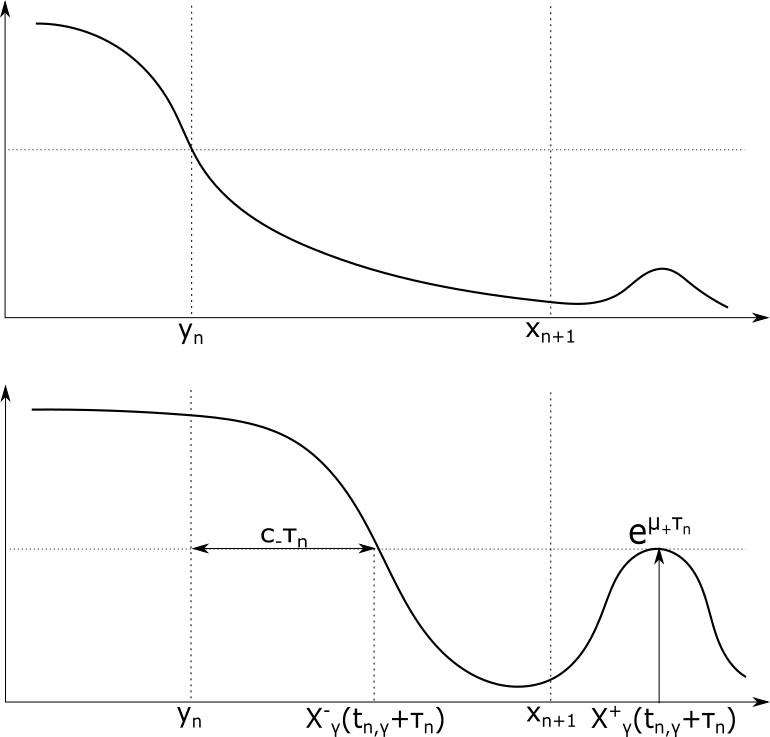}}
\caption{Profiles of the functions $x\mapsto u(t,x)$ at the times $t=t_{n,\gamma}$ and $t=t_{n,\gamma}+\tau_n$, with $c_-=2\sqrt{\mu_-}$; the growing bump has a growth rate $\approx\mu_+$ and its size is magnified by a factor $\approx e^{\mu_+\tau_n}$ between the times $t_{n,\gamma}$ and $t_{n,\gamma}+\tau_n$.}
\label{fig1}
\end{figure}

\begin{proof} 
Let us fix $\gamma\in(0,\Gamma]$ in the proof. Since $u(t,x)\to0$ as $x\to+\infty$ locally uniformly with respect to $t\in[0,+\infty)$, and since $y_n\to+\infty$ as $n\to+\infty$, one has $t_{n,\gamma}\ge1$ for all~$n$ large enough. The Harnack inequality then yields the existence of a constant $C\in(0,1)$ (independent of~$\gamma$ and~$n$) such that
$$u(t_{n,\gamma}+1,\cdot)\geq C\gamma\ \hbox{ in $(y_{n}-1,y_{n})\ $ for all $n$ large enough}.$$

Remember that $R>0$ is given in~\eqref{defReps} and, for each $n\in\N$, denote
\be\label{tau'n}
\tau'_n:=\frac{x_{n+1}-y_n+2R+2}{2\sqrt{\mu_+-\mu_-}}>0.
\ee
Since $0<\e<\mu_-$ and $0<2\mu_-<\mu_+$, one has $0<\sqrt{\mu_--\e}<\sqrt{\mu_+-\mu_-}$ and then
$$x_{n+1}-y_n\ge2\sqrt{\mu_--\e}\,\tau'_n\ \hbox{ for all $n$ large enough},$$
because $x_{n+1}-y_n\to+\infty$ as $n\to+\infty$. With $\theta>0$ as in Lemma~\ref{lem:estdecay}, it follows from Lemma~\ref{lem:estdecay}, applied with $v:=u(t_{n,\gamma}+1+\cdot,y_n+\cdot)$ and $C\gamma\in(0,\Gamma]$ instead of $\gamma$, that
$$\baa{rcl}
u(t_{n,\gamma}+1+\tau'_n,x) & \!\!\ge\!\! & \displaystyle\theta\,C\,\gamma\,e^{(\mu_--\e)\tau'_n}\int_{-1}^0\frac{e^{-(x-y_n-z)^2/(4\tau'_n)}}{\sqrt{4\pi\tau'_n}}\,dz\vspace{3pt}\\
& \!\!\ge\!\! & \displaystyle\frac{\theta\,C\,\gamma\,e^{(\mu_--\e)\tau'_n-(x_{n+1}-y_n+2R+2)^2/(4\tau'_n)}}{\sqrt{4\pi\tau'_n}}\eaa\baa{l}\hbox{for all $n$ large enough}\\ \hbox{and $x\in[x_{n\!+\!1}\!+\!1,x_{n\!+\!1}\!+\!2R\!+\!1]$}.\eaa$$

For each $n$, consider now the function $\underline{u}_n:[0,+\infty)\times[x_{n+1}+1,x_{n+1}+2R+1]\to\R$ defined by
\be\label{underlineun}
\underline{u}_n(t,x):=\underbrace{\frac{\theta\,C\,\gamma\,e^{(\mu_--\e)\tau'_n-(x_{n+1}-y_n+2R+2)^2/(4\tau'_n)}}{\sqrt{4\pi\tau'_n}}}_{=:\alpha_n}\times\cos\!\Big(\frac{\pi(x\!-\!x_{n+1}\!-\!R\!-\!1)}{2R}\Big)\!\times e^{(\mu_+-2\e)t}.
\ee
One has
$$\underline{u}_n(t,x_{n+1}+R+1\pm R)=0<u(t_{n,\gamma}+1+\tau'_n+t,x_{n+1}+R+1\pm R)\ \hbox{ for all $t\ge0$},$$
and
$$\underline{u}_n(0,\cdot)\le u(t_{n,\gamma}+1+\tau'_n,\cdot)\ \hbox{ in $[x_{n+1}+1,x_{n+1}+2R+1]$ for all $n$ large enough}.$$
Furthermore, since $\pi^2/(4R^2)\le\e$ by~\eqref{defReps}, the nonnegative function $\underline{u}_n$ satisfies
$$\partial_t\underline{u}_n-\partial_{xx}\underline{u}_n=\Big(\mu_+-2\e+\frac{\pi^2}{4R^2}\Big)\underline{u}_n\le(\mu_+-\e)\underline{u}_n$$
in $[0,+\infty)\times[x_{n+1}+1,x_{n+1}+2R+1]$. From~\eqref{hypxn}-\eqref{defmu2} and~\eqref{gammaeps}, one has, for all $n$ large enough,
$$\left\{\baa{l}
[x_{n+1}+1,x_{n+1}+2R+1]\subset[x_{n+1}+1,y_{n+1}-1],\vspace{3pt}\\
f(x,s)\ge(\mu(x)-\e)s=(\mu_+-\e)s\hbox{ for all }(x,s)\in[x_{n+1}+1,x_{n+1}+2R+1]\times[0,\Gamma].\eaa\right.$$
Therefore, for all $n$ large enough, one has $\partial_t\underline{u}_n(t,x)\le\partial_{xx}\underline{u}_n(t,x)+f(x,\underline{u}_n(t,x))$ for every $(t,x)\in[0,+\infty)\times[x_{n+1}+1,x_{n+1}+2R+1]$ such that $\underline{u}_n(t,x)\le\Gamma$. Observe also that, for every $\tau>0$, 
$$\max_{[0,\tau]\times[x_{n+1}+1,x_{n+1}+2R+1]}\underline{u}_n=\underline{u}_n(\tau,x_{n+1}+R+1)=\alpha_n\,e^{(\mu_+-2\e)\tau},$$
and that the definitions~\eqref{tau'n}-\eqref{underlineun} of $\tau'_n$ and $\alpha_n$ yield
$$(\mu_--\e)\tau'_n-\frac{(x_{n+1}-y_n+2R+2)^2}{4\tau'_n}\sim\frac{2\mu_--\e-\mu_+}{2\sqrt{\mu_+-\mu_-}}\times(x_{n+1}-y_n)\to-\infty\ \hbox{ as }n\to+\infty$$
and then $\alpha_n\to0$ as $n\to+\infty$. Therefore, for all $n$ large enough, there is a time $\tau''_n>0$ such that
$$\max_{[0,\tau''_n]\times[x_{n+1}+1,x_{n+1}+2R+1]}\underline{u}_n=\alpha_n\,e^{(\mu_+-2\e)\tau''_n}=\gamma\le\Gamma,$$
and
\be\label{tau''n}
\tau''_n\sim\frac{\mu_++\e-2\mu_-}{2(\mu_+-2\e)\sqrt{\mu_+-\mu_-}}\times(x_{n+1}-y_n)\ \hbox{ as }n\to+\infty.
\ee
For all $n$ large enough, the function $\underline{u}_n$ is then a subsolution of the equation satisfied by~$u(t_{n,\gamma}+1+\tau'_n+\cdot,\cdot)$ in $[0,\tau''_n]\times[x_{n+1}+1,x_{n+1}+2R+1]$, and the maximum principle then implies that
$$\underline{u}_n\le u(t_{n,\gamma}+1+\tau'_n+\cdot,\cdot)\ \hbox{ in }[0,\tau''_n]\times[x_{n+1}+1,x_{n+1}+2R+1].$$
In particular, one has
$$u(t_{n,\gamma}+1+\tau'_n+\tau''_n,x_{n+1}+1+R)\ge\underline{u}_n(\tau''_n,x_{n+1}+1+R)=\alpha_n\,e^{(\mu_+-2\e)\tau''_n}=\gamma$$
for all $n$ large enough, hence
$$X^+_\gamma(t_{n,\gamma}+1+\tau'_n+\tau''_n)\ge x_{n+1}+1+R.$$
By setting $\tau_n=1+\tau'_n+\tau''_n$ for all $n$ large enough, and observing that
$$\tau_n\sim\frac{2\mu_+-\e-2\mu_-}{2(\mu_+-2\e)\sqrt{\mu_+-\mu_-}}\times(x_{n+1}-y_n)\ \hbox{ as $n\to+\infty$}$$
by~\eqref{tau'n} and~\eqref{tau''n}, the proof of Lemma~\ref{lem:X+} is thereby complete.
\end{proof}

The next three lemmas are concerned with upper bounds of $u$. We start with an obvious global exponential upper bound following from the definition~\eqref{deff} of $f$ (implying that $f$ satisfies the Fisher-KPP property~\eqref{hyp:fKPP}).

\begin{lem}\label{lem3}
There holds
$$0\le u(t,x)\le\min\big(e^{2\mu_+t-\sqrt{\mu_+}\,x},1\big)\ \hbox{ for all }t\ge0\hbox{ and }x\in\R.$$
\end{lem}

\begin{proof}
From~\eqref{defmu2}-\eqref{deff}, the function $u:[0,+\infty)\times\R\to[0,1]$ satisfies $\partial_tu\le\partial_{xx}u+\mu_+u$ in~$(0,+\infty)\times\R$, while the function $\overline{u}:(t,x)\mapsto\overline{u}(t,x):=e^{2\mu_+t-\sqrt{\mu_+}\,x}$ solves $\partial_t\overline{u}=\partial_{xx}\overline{u}+\mu_+\overline{u}$ in $[0,+\infty)\times\R$, with $ u(0,\cdot)\le\overline{u}(0,\cdot)$ in $\R$. The maximum principle yields the conclusion.
\end{proof}

The second lemma, obtained from Lemma~\ref{lem3} and using the smallness of $\mu$ in the interval $[y_n,x_{n+1}]$, provides an upper bound of $u$ at the position $x_{n+1}$, after the time $s_n$ defined by
\be\label{defsn}
s_n:=\frac{y_n}{2\sqrt{\mu_+}}.
\ee

\begin{lem} \label{lem:estxn}
For every $n\in\N$ and $t\ge s_n$, one has 
\be\label{ineq12}\left\{\baa{rcll}
u(t,x) & \leq & 2\,e^{\mu_+(t-s_n)-\sqrt{\mu_+-\mu_-}(x_{n+1}-y_n)} & \hbox{for all }x\ge x_{n+1},\vspace{3pt}\\
u(t,x) & \leq & e^{2\mu_-(t-s_n)-\sqrt{\mu_-}(x-y_n)}\vspace{3pt}\\
& & +2\,e^{\mu_+(t-s_n)+\sqrt{\mu_+-\mu_-}(x-x_{n+1})-\sqrt{\mu_+-\mu_-}(x_{n+1}-y_n)} & \hbox{for all }y_n\le x \le x_{n+1}.\eaa\right.
\ee 
\end{lem}

\begin{proof}
Let us fix any integer $n$ throughout the proof. Let us first define, for $t\ge s_n$ and $x\ge y_n$:
$$\overline{u}_n(t,x):=\left\{\begin{array}{ll} 
\!\!e^{\mu_+(t-s_n)-\sqrt{\mu_+\!-\!\mu_-}(x-y_n)}\!+\!e^{\mu_+(t-s_n)\!+\!\sqrt{\mu_+\!-\!\mu_-}(x-x_{n\!+\!1})\!-\!\sqrt{\mu_+\!-\!\mu_-}(x_{n\!+\!1}-y_n)} & \!\hbox{if }y_n\!\le\!x\!<\!x_{n\!+\!1},\vspace{3pt}\\
\!\!2\,e^{\mu_+(t-s_n)-\sqrt{\mu_+-\mu_-}(x_{n+1}-y_n)} & \!\hbox{if }x\ge x_{n+1}.\end{array}\right.$$ 
This function $\overline{u}_n$ is of class $C^1$ (with respect to the variables $(t,x)$) in $[s_n,+\infty)\times[y_n,+\infty)$ and of class $C^2$ with respect to $x$ in $[s_n,+\infty)\times([y_n,+\infty)\!\setminus\!\{x_{n+1}\})$. We also claim that it is a supersolution of the equation~\eqref{eq:main} satisfied by $u$, in $[s_n,+\infty)\times[y_n,+\infty)$. Indeed, first of all, by~\eqref{defsn} and Lemma~\ref{lem3}, one has $u(s_n,x)\le e^{2\mu_+s_n-\sqrt{\mu_+}x}=e^{-\sqrt{\mu_+}(x-y_n)}$ for all $x\ge y_n$, hence
$$u(s_n,x)\le e^{-\sqrt{\mu_+-\mu_-}(x-y_n)}\le \overline{u}_n(s_n,x)\ \hbox{ for all $x\ge y_n$}.$$
Furthermore,
$$u(t,y_n)\le 1\le \overline{u}_n(t,y_n)\ \hbox{ for all $t\ge s_n$}.$$
Lastly, since $\mu=\mu_-$ in $[y_n,x_{n+1}]$ and $\mu\le\mu_+$ in $\R$ by~\eqref{defmu2}, it is easy to see that
$$\partial_t\overline{u}_n(t,x)-\partial_{xx}\overline{u}_n(t,x)=\mu_-\overline{u}_n(t,x)\ge f(x,\overline{u}_n(t,x))$$
for all $(t,x)\in[s_n,+\infty)\times[y_n,x_{n+1})$ such that $\overline{u}_n(t,x)\le1$, whereas
$$\partial_t\overline{u}_n(t,x)-\partial_{xx}\overline{u}_n(t,x)=\mu_+\overline{u}_n(t,x)\ge f(x,\overline{u}_n(t,x))$$
for all $(t,x)\in[s_n,+\infty)\times(x_{n+1},+\infty)$ such that $\overline{u}_n(t,x)\le1$. Therefore, remembering also that $u\le1$, it follows from the maximum principle that
$$u\le\min(\overline{u}_n,1)\ \hbox{ in $[s_n,+\infty)\times[y_n,+\infty)$}.$$
In particular, at $(t,x)$ with any $t\ge s_n$ and $x\ge x_{n+1}$, the inequality $u(t,x_{n+1})\le \overline{u}_n(t,x)$ yields the first inequality in~\eqref{ineq12}.

To get the second one, we now define, for $t\ge s_n$ and $x\in\R$:
\be\label{defvn}
\overline{v}_n(t,x):=e^{2\mu_-(t-s_n)-\sqrt{\mu_-}(x-y_n)}+2\,e^{\mu_+(t-s_n)+\sqrt{\mu_+-\mu_-}(x-x_{n+1})-\sqrt{\mu_+-\mu_-}(x_{n+1}-y_n)}.
\ee
The function $\overline{v}_n$ is of class $C^2$ in $[s_{n},+\infty)\times\R$ and it obeys
$$\partial_{t}\overline{v}_n-\partial_{xx}\overline{v}_n=\mu_{-}\overline{v}_n\ \hbox{ in }[s_{n},+\infty)\times\R.$$
Hence, the definitions~\eqref{defmu2}-\eqref{deff} of $\mu$ and $f$ imply that
$$\partial_{t}\overline{v}_n(t,x)-\partial_{xx}\overline{v}_n(t,x)\ge f(x,\overline{v}_n(t,x))$$
for every $(t,x)\in[s_{n},+\infty)\times[y_n,x_{n+1}]$ such that $\overline{v}_n(t,x)\le1$. Furthermore, Lemma~\ref{lem3} and~\eqref{defsn} imply that, for every~$x\in\R$, $u(s_n,x) \leq \min\big(e^{-\sqrt{\mu_+}(x-y_n)},1\big)$, hence $u(s_{n},x)\leq\overline{v}_n(s_{n},x) $ for all $x\ge y_n$ (remember that $\sqrt{\mu_-}<\sqrt{\mu_+}$). Lastly, $u(t,y_n)\le 1\le\overline{v}_n(t,y_n)$ for all $t\ge s_n$, and $u(t,x_{n+1})\le\overline{v}_n(t,x_{n+1})$ for all $t\ge s_n$ from the first inequality in~\eqref{ineq12}. Remembering that $u\le1$, it then follows from the maximum principle that
$$u(t,x)\leq\min(\overline{v}_n(t,x),1)\ \hbox{ for all $(t,x)\in[s_n,+\infty)\times[y_{n},x_{n+1}]$},$$
which completes the proof of Lemma~\ref{lem:estxn}.
\end{proof}

Lemma~\ref{lem:estxn} then provides an upper bound of $X^-_\gamma(t)$ for $t$ belonging to some interval on the right of $s_n$, as the following lemma shows.

\begin{lem}\label{lemX-}
Let $\Gamma\in(0,1)$ be given in~\eqref{gammaeps}. For every $\gamma\in(0,\Gamma]$ and $\ell>-\ln(\gamma)/\sqrt{\mu_-}$, there exists a constant $M=M_{\gamma,\ell}$ such that, for all $n$ large enough,
$$X^-_\gamma(t)\leq\ell+y_{n}+2\sqrt{\mu_{-}}\,(t-s_n)\ \hbox{ for all }s_n\le t\leq s_n+M+\frac{2\sqrt{\mu_+-\mu_-}}{\mu_{+}+2\sqrt{\mu_-(\mu_+-\mu_-)}}\times(x_{n+1}-y_n),$$
with $s_n$ defined by~\eqref{defsn}.
\end{lem}

\begin{proof}
Let $\gamma\in(0,\Gamma]$ and $\ell>-\ln(\gamma)/\sqrt{\mu_-}$ be fixed throughout the proof (notice that $\ell>0$, since $0<\gamma\le\Gamma<1$). Denote
\be\label{defM}
M=M_{\gamma,\ell}:=\frac{\ln(\gamma-e^{-\ell\sqrt{\mu_-}})-\ln 2-\ell\sqrt{\mu_+-\mu_-}}{\mu_{+}+2\sqrt{\mu_-(\mu_+-\mu_-)}},
\ee
which is well defined since $\ell>-\ln(\gamma)/\sqrt{\mu_-}$. As $2\sqrt{\mu_-(\mu_+-\mu_-)}<\mu_+$ by~\eqref{ineqmupm}, one has
$$\frac{4\sqrt{\mu_-(\mu_+-\mu_-)}}{\mu_++2\sqrt{\mu_-(\mu_+-\mu_-)}}<1,$$
hence there is $n_0\in\N$ such that
\be\label{ellM}\left\{\baa{l}
\displaystyle\ell+2M\sqrt{\mu_-}+\frac{4\sqrt{\mu_-(\mu_+-\mu_-)}}{\mu_++2\sqrt{\mu_-(\mu_+-\mu_-)}}\times(x_{n+1}-y_n)<x_{n+1}-y_n\vspace{3pt}\\
\displaystyle M+\frac{2\sqrt{\mu_+-\mu_-}}{\mu_{+}+2\sqrt{\mu_-(\mu_+-\mu_-)}}\times(x_{n+1}-y_n)\ge0\eaa\right.
\ee
for all $n\ge n_0$.

Consider finally any $n\ge n_0$ and any $t$ such that
\be\label{choicet}
s_n\le t\leq s_n+M+\frac{2\sqrt{\mu_+-\mu_-}}{\mu_{+}+2\sqrt{\mu_-(\mu_+-\mu_-)}}\times(x_{n+1}-y_n).
\ee
One has $y_n<\ell+y_n+2\sqrt{\mu_{-}}\,(t-s_n)<x_{n+1}$ from~\eqref{ellM} and the positivity of $\ell$, hence the second inequality in the conclusion~\eqref{ineq12} of Lemma~\ref{lem:estxn}, together with the definition~\eqref{defvn}, yields
$$\baa{rcl}
u(t,\ell\!+\!y_{n}\!+\!2\sqrt{\mu_{-}}(t\!-\!s_n)) & \!\!\leq\!\! & \overline{v}_n(t,\ell+y_n+2\sqrt{\mu_-}(t-s_n))\vspace{3pt}\\
& \!\!\le\!\! & e^{-\ell\sqrt{\mu_-}}\!+\!2\,e^{-2\sqrt{\mu_+-\mu_-}(x_{n+1}-y_n)+[\mu_{+}+2\sqrt{\mu_-(\mu_+-\mu_-)}](t-s_n)+\ell\sqrt{\mu_+-\mu_-}}\eaa$$
and finally $u(t,\ell+y_{n}+2\sqrt{\mu_{-}}\,(t-s_n))\leq\gamma$ from~\eqref{defM} and~\eqref{choicet}. This implies that
$$X^-_\gamma(t)\leq\ell+y_{n}+2\sqrt{\mu_{-}}\,(t-s_n)$$
and the proof of Lemma~\ref{lemX-} is thereby complete.
\end{proof}

With the above lemmas in hand, we can complete the proof of Theorem~\ref{th1}.\hfill\break

\noindent{\it{Proof of Theorem~$\ref{th1}$}}. Let $(x_n)_{n\in\N}$, $(y_n)_{n\in\N}$, $\mu$ and $f$ be as in~\eqref{mupm}-\eqref{deff}, and let $\e_0\in(0,\mu_-)$ be such that~\eqref{defepsilon} holds for all $\e\in(0,\e_0]$. Consider then any such $\e\in(0,\e_0]$, and let $R>0$ and $\Gamma\in(0,1)$ be as in~\eqref{defReps}-\eqref{gammaeps}.

Consider first in this paragraph any $\gamma\in(0,\Gamma]$. Fix any
$$\ell>-\frac{\ln(\gamma)}{\sqrt{\mu_-}}$$
and let $M\in\R$ be as in Lemma~\ref{lemX-}. With $t_{n,\gamma}>0$ given as in~\eqref{tngamma}, Lemma~\ref{lem:X+} implies that
\be\label{X+gamma}
X^+_\gamma(t_{n,\gamma}+\tau_n)\ge x_{n+1}+1+R\ \hbox{ for all $n$ large enough},
\ee
with $\tau_n>0$ satisfying~\eqref{eqtaun}. Next, on the one hand, since $t_{n,\gamma}>0$ and since~$s_n$ and~$\tau_n$ are respectively of the order $y_n$ and $x_{n+1}-y_n$ (as $n\to+\infty$) from~\eqref{defsn} and~\eqref{eqtaun}, the condition $y_n=o(x_{n+1}-y_n)$ as $n\to+\infty$ implies that
$$s_n\le t_{n,\gamma}+\tau_n\ \hbox{ for all $n$ large enough}.$$
On the other hand, by picking any $c$ such that $0<c<2\sqrt{\mu_{-}}$, it follows from~\eqref{uto1} and $\lim_{n\to+\infty}y_n=+\infty$ that $0<t_{n,\gamma}<y_n/c$ for all $n$ large enough. As $2s_n\sqrt{\mu_{+}}=y_{n}$, one deduces from~\eqref{defepsilon},~\eqref{eqtaun} and $y_n=o(x_{n+1}-y_n)$ that
$$t_{n,\gamma}+\tau_n-s_n\le M+\frac{2\sqrt{\mu_+-\mu_-}}{\mu_{+}+2\sqrt{\mu_-(\mu_+-\mu_-)}}\times(x_{n+1}-y_n)\ \hbox{ for all $n$ large enough}.$$
One then infers from Lemma~\ref{lemX-} that
$$X^{-}_\gamma(t_{n,\gamma}+\tau_n)\leq \ell+y_{n}+2\sqrt{\mu_{-}}\,(t_{n,\gamma}+\tau_n-s_n)\ \hbox{ for all $n$ large enough}.$$
Together with~\eqref{X+gamma}, it follows that
$$I_\gamma(t_{n,\gamma}+\tau_n)=X^+_\gamma(t_{n,\gamma}+\tau_n)-X^-_\gamma(t_{n,\gamma}+\tau_n)\ge 1+R-\ell+x_{n+1}-y_n-2\sqrt{\mu_{-}}\,(t_{n,\gamma}+\tau_n-s_n)$$
for all $n$ large enough. Remember that $y_n=o(x_{n+1})=o(x_{n+1}-y_n)$ as $n\to+\infty$, while $0<t_{n,\gamma}<y_n/c$ (for all $n$ large enough) and $s_n=y_n/(2\sqrt{\mu_+})$ then yield $t_{n,\gamma}=o(x_{n+1}-y_n)$ and $s_n=o(x_{n+1}-y_n)$ as $n\to+\infty$. Since $t_{n,\gamma}+\tau_n\to+\infty$ as $n\to+\infty$, one then gets from~\eqref{eqtaun} and~\eqref{defepsilon} that
$$\limsup_{t\to+\infty}\frac{I_\gamma(t)}{t}\ge\limsup_{n\to+\infty}\frac{I_\gamma(t_{n,\gamma}+\tau_n)}{t_{n,\gamma}+\tau_n}\ge\frac{2(\mu_+-2\e)\sqrt{\mu_+-\mu_-}}{2\mu_+-\e-2\mu_-}-2\sqrt{\mu_-}>0.$$
Since the above inequality is valid for every $\gamma\in(0,\Gamma]$, and since $\e$ can be arbitrary in $(0,\e_0]$, it follows that
$$\liminf_{\gamma\to 0^{+}}\Big(\limsup_{t\to +\infty} \frac{I_{\gamma}(t)}{t}\Big)\geq\frac{2\mu_+\sqrt{\mu_+-\mu_-}}{2\mu_{+}-2\mu_{-}}-2\sqrt{\mu_-}=\frac{\mu_+ }{\sqrt{\mu_{+}-\mu_{-}}}-2\sqrt{\mu_-}>0,$$
that is,~\eqref{main1} has been proved.

In this remaining part of the proof of Theorem~\ref{th1}, consider any $\gamma\in(0,1)$ and let us show~\eqref{main2}. On the one hand,~\eqref{uto1} implies that
\be\label{liminf-}
\liminf_{t\to +\infty}\frac{X^{-}_\gamma(t)}{t}\geq 2\sqrt{\mu_{-}}.
\ee
On the other hand, consider any $\sigma>2\sqrt{\mu_-}$. Since $0<y_n<x_{n+1}$ while $y_n\to+\infty$ and $y_n=o(x_{n+1})$ as $n\to+\infty$ by~\eqref{hypxn}, one has $y_n<\sqrt{y_nx_{n+1}}<x_{n+1}$ and
$$+\infty\leftarrow\sqrt{y_nx_{n+1}}-y_n\sim\sqrt{y_nx_{n+1}}=o(x_{n+1})=o(x_{n+1}-y_n)\ \hbox{ as $n\to+\infty$}.$$
Therefore, owing to the definition~\eqref{defvn} of $\overline{v}_n$ and to $\mu_+>\mu_->0$ and $\sigma>2\sqrt{\mu_-}$, there holds
$$\baa{l}
\displaystyle\overline{v}_n\Big(s_n+\frac{\sqrt{y_nx_{n+1}}}{\sigma},\sqrt{y_nx_{n+1}}\Big)\vspace{3pt}\\
=e^{2\mu_-\sqrt{y_nx_{n+1}}/\sigma-\sqrt{\mu_-}(\sqrt{y_nx_{n+1}}-y_n)}\!+\!2\,e^{\mu_+\sqrt{y_nx_{n+1}}/\sigma+\sqrt{\mu_+-\mu_-}(\sqrt{y_nx_{n+1}}-x_{n+1})-\sqrt{\mu_+-\mu_-}(x_{n+1}-y_n)}\to0\eaa$$
as $n\to+\infty$. Similarly,
$$\overline{v}_n\Big(s_n+\frac{\sqrt{y_nx_{n+1}}}{\sigma},x_{n+1}\Big)=e^{2\mu_-\sqrt{y_nx_{n+1}}/\sigma-\sqrt{\mu_-}(x_{n+1}-y_n)}\!+\!2\,e^{\mu_+\sqrt{y_nx_{n+1}}/\sigma-\sqrt{\mu_+-\mu_-}(x_{n+1}-y_n)}\to0$$
as $n\to+\infty$. Since each function $\overline{v}_n(s_n+\sqrt{y_nx_{n+1}}/\sigma,\cdot)$ is convex in $\R$ by definition, one gets that
$$\max_{[\sqrt{y_nx_{n+1}},x_{n+1}]}\overline{v}_n\Big(s_n+\frac{\sqrt{y_nx_{n+1}}}{\sigma},\cdot\Big)\to0\ \hbox{ as $n\to+\infty$}.$$
Together with the second inequality of~\eqref{ineq12} in Lemma~\ref{lem:estxn} and the nonnegativity of $u$, it follows that $\max_{[\sqrt{y_nx_{n+1}},x_{n+1}]}u(s_n+\sqrt{y_nx_{n+1}}/\sigma,\cdot)\to0$ as $n\to+\infty$. But the first inequa\-lity of~\eqref{ineq12} in Lemma~\ref{lem:estxn} and the comparison $\sqrt{y_nx_{n+1}}=o(x_{n+1}-y_n)$ also imply that $\max_{[x_{n+1},+\infty)}u(s_n+\sqrt{y_nx_{n+1}}/\sigma,\cdot)\to0$ as $n\to+\infty$.\footnote{Notice that, for each $t>0$ and $a\in\R$, $\max_{[a,+\infty)}u(t,\cdot)$ is well defined since $u(t,\cdot)$ is positive and continuous in~$\R$, and $u(t,+\infty)=0$.} Finally,
$$\max_{[\sqrt{y_nx_{n+1}},+\infty)}u\Big(s_n+\frac{\sqrt{y_nx_{n+1}}}{\sigma},\cdot\Big)\to0\ \hbox{ as $n\to+\infty$},$$
hence
$$X^+_\gamma\Big(s_n+\frac{\sqrt{y_nx_{n+1}}}{\sigma}\Big)<\sqrt{y_nx_{n+1}}$$
for all $n$ large enough. Since $s_n+\sqrt{y_nx_{n+1}}/\sigma\sim\sqrt{y_nx_{n+1}}/\sigma\to+\infty$ as $n\to+\infty$ by~\eqref{hypxn} and~\eqref{defsn}, one gets
$$\liminf_{t\to +\infty}\frac{X^{+}_\gamma(t)}{t}\le\liminf_{n\to +\infty}\frac{\displaystyle X^{+}_\gamma\Big(s_n+\frac{\sqrt{y_nx_{n+1}}}{\sigma}\Big)}{\displaystyle s_n+\frac{\sqrt{y_nx_{n+1}}}{\sigma}}\le\sigma.$$
As $\sigma$ was arbitrary in $(2\sqrt{\mu_-},+\infty)$, one obtains
$$\liminf_{t\to +\infty}\frac{X^{+}_\gamma(t)}{t}\leq2\sqrt{\mu_{-}}.$$
Since $I_\gamma(t)=X^+_\gamma(t)-X^-_\gamma(t)\ge0$ for all $t>0$, one concludes from the last inequality and~\eqref{liminf-} that
$$\liminf_{t\to +\infty}\frac{I_{\gamma}(t)}{t}=0.$$

Lastly, Lemma~\ref{lem3} implies that
$$\limsup_{t\to +\infty}\frac{X^{+}_\gamma(t)}{t}\leq 2\sqrt{\mu_{+}},$$
which, together with~\eqref{liminf-}, yields
$$\limsup_{t\to +\infty}\frac{I_{\gamma}(t)}{t}\le2(\sqrt{\mu_+}-\sqrt{\mu_-})<+\infty.$$
The proof of Theorem~\ref{th1} is thereby complete.\hfill$\Box$

\bigskip

\noindent {\it Proof of Corollary $\ref{cor1}$.} Assume that there exists a generalized transition front $U$. Then 
for every $\gamma\in(0,1)$, there is a real number~$C_\gamma\ge0$ such that, for any $(t,x)\in\R\times\R$ satisfying $U(t,x)=\gamma$, there holds $U(t,y)>\gamma$ for every $y<x-C_\gamma$ and $U(t,y)<\gamma$ for every $y>x+C_\gamma$. 

As $\mu_->0$, we refer to~\cite[Theorem~7.1]{BHR} which implies that, for any continuous function $v_0:\R\to[0,1]$ with $\|v_0\|_{L^\infty(\R)}>0$, the solution $v$ of~\eqref{eq:main} with initial condition $v_0$ satisfies $v(t,x)\to1$ as $t\to+\infty$ locally uniformly in $x\in\R$. In particular, $U(t,x)\to1$ as $t\to+\infty$ locally uniformly in $x\in\R$, and then~$U(t,x)\to0$ as~$t\to-\infty$ locally uniformly in $x\in\R$ (since $0<U<1$ and $\partial_xU$ is bounded in $\R\times\R$ from standard parabolic estimates). 

Consider now the solution $u$ of~\eqref{eq:main} and~\eqref{hyp:u0} associated with the initial condition $u_{0}=\1_{(-\infty,0)}$, and pick any~$\gamma\in(0,1)$, $t>0$ and $x\in\R$ such that
$$u(t,x)=\gamma.$$
From the previous observation and the continuity of $U$, there is $t_0\in\R$ such that $U(t_0,x)=\gamma$. Since $u_0(y)=1>U(t_0-t,y)$ for all $y<0$ and $u_0(y)=0<U(t_0-t,y)$ for all $y>0$, it follows that
$$u(t,y)>U(t_0,y)\hbox{ for all $y<x$ and~$u(t,y)<U(t_0,y)$ for all $y>x$}$$
(in other words, since $u_0$ is steeper than $U(t_0-t,\cdot)$, $u(t,\cdot)$ is steeper than $U(t_0,\cdot)$ as well, see~\cite{Angenent, DGM}). Therefore, $u(t,y)>U(t_0,y)>\gamma$ for all $y<x-C_\gamma$ and $u(t,y)<U(t_0,y)<\gamma$ for all $y>x-C_\gamma$, which yields $I_\gamma(t)<2C_\gamma$ and then
$$\sup_{t>0}I_\gamma(t)\le2C_\gamma<+\infty.$$
The proof of Corollary~\ref{cor1} is thereby complete.\hfill$\Box$


\end{document}